\documentclass[11pt,reqno]{amsart}

\usepackage{latexsym} 
  \usepackage[all]{xy}
  \usepackage{amsfonts} 
  \usepackage{amsthm} 
  \usepackage{amsmath} 
  \usepackage{amssymb}
  \usepackage{pifont}  
  \usepackage{enumerate}
 \newtheorem{thm}{Theorem}[section]
 \newtheorem{cor}[thm]{Corollary}
 
 \newtheorem{prop}[thm]{Proposition}
 \theoremstyle{definition}
 \newtheorem{defn}[thm]{Definition}
 \theoremstyle{remark}
 \newtheorem{rem}[thm]{Remark}
 
 \numberwithin{equation}{section}
 
 \def\bs{\begin{statement}}
\def\es{\end{statement}}
  \swapnumbers
  \newtheorem{statement}[thm]{}

  \newcounter{zlist} 
  \newenvironment{zlist}{\begin{list}{(\arabic{zlist})}{ 
  \usecounter{zlist}\leftmargin2.5em\labelwidth2em\labelsep0.5em 
  \topsep0.6ex%\itemsep0.3ex plus0.2ex minus0.3ex 
  \parsep0.3ex plus0.2ex minus0.1ex}}{\end{list}}

  \newcounter{blist} 
  \newenvironment{blist}{\begin{list}{(\alph{blist})}{ 
  \usecounter{blist}\leftmargin2.5em\labelwidth2em\labelsep0.5em 
  \topsep0.6ex %\itemsep0.3ex plus0.2ex minus0.3ex 
  \parsep0.3ex plus0.2ex minus0.1ex}}{\end{list}} 

  \newcounter{rlist}

\newcommand{\Cc}{\mathcal{C}}

\def\*C{{}^*\hspace*{-1pt}{\Cc}}

\def\text#1{{\rm {\rm #1}}}

 \def\1{\mathbf{1}}

  \def\la{\triangleright}

\def\id{\mathrm{id}}
\def\pi {\mathrm{pi}}
\def\di{\diamond}

\begin{document}

\title{Towards semi-trusses}

\author{Tomasz Brzezi\'nski}

\address{
Department of Mathematics, Swansea University, 
  Swansea SA2 8PP, U.K.\ \newline \indent
Department of Mathematics, University of Bia{\l}ystok, K.\ Cio{\l}kowskiego  1M,
15-245 Bia\-{\l}ys\-tok, Poland}

\email{T.Brzezinski@swansea.ac.uk}



\date{November 2017}

\begin{abstract}
A general or {\em truss} distributive laws between two associative operations on the same set are studied for cancellative and inverse semigroups.
\end{abstract}

\maketitle

\section{Introduction}
The appearance of {\em braces} in \cite{Rum:bra}, \cite{CedJes:bra} and later of {\em skew braces} in \cite{GuaVen:ske} was motivated by a search for bijective solutions to the set-theoretic Yang-Baxter equation \cite{Dri:uns}.  A  {\em (skew left) brace} is a set $A$ together with two group operations $\diamond$ and $\circ$ related by the {\em brace distributive law}, for all $a,b,c\in A$,
\begin{equation}\label{brace.law}
a\circ (b\diamond c) = (a\circ b) \diamond a^\diamond\diamond(a\circ c),
\end{equation}
where $a^\diamond$ is the inverse of $a$ with respect to $\diamond$. Soon after their introduction a deep connection of braces with group theory has been observed, thus resulting in a rapidly growing literature devoted to braces, their applications and generalisations; see e.g.\ \cite{Rum:mod},  \cite{Bac:cla}, \cite{Bac:cou}, \cite{SmoVen:ske}, \cite{LebVen:coh} for the analysis of the structure of braces,  \cite{CedGat:Yan}, \cite{BacCed:sol}, \cite{BacCed:asy}, \cite{Bac:sol}, \cite{Gat:set}, \cite{GucGug:Yan}, \cite{BacCed:char}  for the application of braces to the Yang-Baxter equation, 
\cite{Rum:dyn}, \cite{CedJes:nil}, \cite{CedGat:bra}, \cite{CatCol:reg}, \cite{Smo:Eng} for the connection between braces and group theory, and  \cite{AngGal:Hop}, \cite{Ago:con} for extension to Hopf algebras. 
 
In an attempt  at understanding the nature of the brace dristributive law, in particular in relation to the ring distributive law, 
the notion of a {\em (skew left) truss}  has been introduced in \cite{Brz:tru}. A truss is a set $A$ together with a group operation $\diamond$ and a semigroup operation $\circ$ which distributes over the {\em heap ternary operation}
\begin{equation}\label{ternary}
[-,-,-]: A\times A\times A \to A, \qquad [a,b,c] := a\diamond b^\diamond\diamond c,
\end{equation}
i.e.\
\begin{equation}\label{dist.heap}
a\circ \left (b\diamond c^\diamond\diamond d\right) = (a\circ b)\diamond (a\circ c)^\diamond\diamond(a\circ d), \qquad \mbox{for all $a,b,c,d\in A$}.
\end{equation}
The conceptually cleanest formulation of truss distributive law \eqref{dist.heap} has several  equivalent formulations in terms of auxiliary functions $\sigma: A\to A$ or $\lambda, \mu, \kappa, \hat\kappa: A\times A \to A$, namely, for all $a,b,c\in A$,
\begin{subequations}\label{truss.law}
\begin{equation}\label{sigma.law}
a\circ (b\diamond c) = (a\circ b) \diamond \sigma(a)^\diamond\diamond(a\circ c),
\end{equation}
\begin{equation}\label{lambda.law}
a\circ(b\diamond c) = (a\circ b)\diamond \lambda (a, c),
\end{equation}
\begin{equation}\label{mu.law}
a\circ(b\diamond c) = \mu (a,b)\diamond (a\circ c) , 
\end{equation}
\begin{equation}\label{kappa.law}
a\circ (b\diamond c) = \kappa(a,b)\diamond \hat\kappa(a,c).
\end{equation}
\end{subequations}
The formulation \eqref{sigma.law} makes it clear that the truss distributive law interpolates between ring and brace distributive laws. The key observations made in \cite{Brz:tru} include not only the equivalence of conditions \eqref{dist.heap} with any one of \eqref{truss.law}, but also that the maps $\lambda$ and $\mu$ encode left actions of $(A,\circ)$ on $(A,\diamond)$ by group endomorphisms which are important ingredients of the brace structures, that $\sigma$ displays the cocycle-type behaviour and that if $\sigma$ is bijective a new operation $\bullet$ replacing $\circ$ can be introduced, which distributes over $\diamond$ according to the brace law \eqref{brace.law}. In particular, if $(A,\circ)$ is a group, then $(A,\diamond,\bullet)$ is a brace.

This paper is motivated by an extension of the notion of a brace to that of a semi-brace in \cite{CatCol:sem}. A set $A$ together with binary operations $\diamond$ and $\circ$ such that $(A,\diamond)$ is a left-cancellative semigroup and $(A,\circ)$ is a group is called a {\em semi-brace}, provided, for all $a,b,c\in A$,
\begin{equation}\label{semi-brace}
a\circ (b\diamond c) = (a\circ b)\diamond \left(a\circ\left(a^\circ\diamond c\right)\right),
\end{equation}
where $a^\circ$ denotes the inverse of $a$ in $(A,\circ)$. Clearly, \eqref{semi-brace} is a special case of \eqref{lambda.law}, which leads to the definition of the main object studied in the present paper:

\begin{defn}\label{def.semi-truss}
A set $A$ together with two associative binary operations $\diamond$, $\circ$ is called a {\em left semi-truss} if there exists function $\lambda: A\times A\to A$ such that, for all $a,b,c\in A$,
\begin{equation}\label{semi-truss} 
a\circ(b\diamond c) = (a\circ b)\diamond \lambda (a, c).
\end{equation}
A left semi-truss is denoted by $(A,\diamond,\circ,\lambda)$.
\end{defn}
As a matter of course one cannot expect that the distributive law \eqref{semi-truss} has equivalent formulations such as in Eqs.\  \eqref{truss.law} or \eqref{dist.heap}, or that $\lambda$ has additional properties (such as being an action) if $\di$ is simply an associative operation; more restrictive properties must be requested of $\di$. We look at two such restrictions. First we assume that $(A,\di)$ is a left cancellative semigroup. In this case $\lambda$ is an action: just as was the case for the trusses, $(A,\circ)$ acts on $(A,\di)$ by endomorphisms through $\lambda$; see Proposition~\ref{prop.act}. Although there are no inverses in left cancellative semigroups, for elements related by a natural pre-order $a\leq b$ there exists a unique element resembling $a^\di\di b$. Using this element we can at least write a distributive law of the type \eqref{dist.heap}; we show that semi-truss with a left cancellative $(A,\di)$ satisfies this law.  If, furthermore, $(A,\di)$ has an idempotent, then this is an equivalent formulation of \eqref{dist.heap}, and there is a third equivalent formulation through a function $\sigma:A\to A$ similar to that in \eqref{sigma.law}; see Proposition~\ref{prop.equiv}. Also in this case, the function $\sigma$ satisfies the cocycle-type property, and it is bijective whenever $(A,\circ)$ is a group. If this happens then the semi-truss can be converted into a semi-brace; see Proposition~\ref{prop.sigma} and Proposition~\ref{prop.semi-brace}.

The second situation we study is when $(A,\di)$ is an inverse semigroup. Since every element $a$ of $(A,\di)$  has a unique inverse $a^\di$, one can genuinely consider all variety of formulations of the truss distributive law. We show that the ternary formulation \eqref{dist.heap} implies laws similar to \eqref{sigma.law} (albeit $\sigma$ is now a function $A\times A\to A$), which in turn imply \eqref{semi-truss} (and, through a separate route \eqref{mu.law}); see Proposition~\ref{prop.inverse}. In the opposite direction rather than obtaining equalities between two sides of a distributive law, one obtains  an inequality (through the natural partial order defined on an inverse semigroup). In a more restrictive situation (such as the one coming from the law \eqref{dist.heap}) we show that, for all $a\in A$, $\lambda (a,-):A\to A$ commutes with the inverse function on $(A,\di)$ and it is an endomorphism of the semi-lattice  $E(A,\di)$ of idempotents of $(A,\di)$.

\section{Semi-trusses for cancellative semigroups}\label{sec.cancel}
Recall that a semigroup $(A,\diamond)$ is said to be {\em left cancellative} if it has the left cancellation property, i.e.,  for all $a,b,c\in A$, $a\diamond b = a\diamond c$ implies that $b=c$. We introduce a pre-order $\leq$  on $(A,\diamond)$ by setting
\begin{equation}\label{leq.cancel}
a\leq b \quad \mbox{if and only if there exists $c\in A$ such that} \quad a\diamond c = b.
\end{equation}
By the left cancellation property if $a\leq b$ then the element $c$, $a \diamond c = b$ is unique and is denoted by $\underline{a^\diamond \diamond b}$. This notation is intended to capture the property
\begin{equation}\label{cancel.inverse}
a\diamond \underline{a^\diamond \diamond b} =b,
\end{equation}
that fully characterises $\underline{a^\diamond \diamond b}$, while the underlining should help to remember that we are dealing with a single element relating two elements rather than with the product of two separate elements. 

Observe that in a left cancellative semigroup every idempotent is a left identity.

\begin{prop}\label{prop.act}
Let $(A,\diamond,\circ,\lambda)$ be a left semi-truss. If $(A,\diamond)$ is a left cancellative semigroup, then, for all $a,b,c\in A$,
\begin{subequations}\label{act}
\begin{equation}\label{act1}
\lambda (a\circ b, c) = \lambda(a, \lambda(b,c)),
\end{equation}
\begin{equation}\label{act2}
\lambda (a, b\diamond c) = \lambda(a,b)\diamond \lambda (a,c).
\end{equation}
\end{subequations}
Furthermore, if $(A,\circ)$ has a left identity $n$, then, for all $a\in A$, 
\begin{equation}\label{act3}
\lambda(n,a) = a.
\end{equation}
\end{prop}
\begin{proof}
All these statements are consequences of the semi-truss law \eqref{semi-truss} and the left cancellation law for $\diamond$. Specifically, to prove \eqref{act1}, let us take any $a,b,c,d\in A$. Then, on one hand
$$
(a\circ b)\circ  (c\diamond d) = (a\circ b\circ c)\diamond \lambda(a\circ b,d),
$$
while on the other
\begin{align*}
(a\circ b)\circ (c\diamond d) = (a\circ (b\circ (c\diamond d)))  &= a\circ \left( (b\circ c)\diamond \lambda (b,d)\right) \\
&= (a\circ b\circ c)\diamond \lambda\left(a,\lambda(b,d)\right).
\end{align*}
Hence property \eqref{act1} follows by the cancellation law. Similarly, on one hand
$$
a\circ (b\diamond c\diamond d) = (a\circ b)\diamond \lambda(a,c\diamond d),
$$
while on the other
\begin{align*}
a\circ (b\diamond c\diamond d) &= (a\circ (b\diamond c))\diamond \lambda (a,d)\\
&= (a\circ b)\diamond \lambda(a,c)\diamond \lambda(a,d),
\end{align*}
which yields \eqref{act2} by cancellation laws.

Finally, for all $a\in A$,
$$
n\diamond a = n\circ (n\diamond a) = (n\circ n) \diamond \lambda(n,a) = n\diamond \lambda(n,a),
$$
and \eqref{act3} follows by the left cancellation law.
\end{proof}

Proposition~\ref{prop.act} thus states that if $(A,\diamond)$ is a left cancellative semigroup, then  $(A,\circ)$ acts on $(A,\diamond)$  by endomorphisms. This is in perfect accord with the similar property of skew trusses \cite[Theorem~2.9]{Brz:tru}, and therefore also skew braces. We denote this action by
\begin{equation}\label{notation.act}
a\la b := \lambda(a,b).
\end{equation}

\begin{prop}\label{prop.equiv}
Let $\diamond$ and $\circ$ be associative binary operations on a set $A$ such that $(A,\diamond)$ is a left cancellative semigroup. Consider the following statements:
\begin{zlist}
\item There exists  function $\sigma:A\to A$ such that, for all $a,b,c\in A$, $\sigma(a)\leq a\circ c$ and
\begin{equation}\label{semi.brace}
a\circ (b\diamond c) = (a\circ b) \diamond \underline{\sigma(a)^\diamond\diamond(a\circ c)}.
\end{equation}
\item There exists function $\lambda: A\times A \to A$ so that $(A,\diamond,\circ,\lambda)$ is a left semi-truss.
\item For all $a,b,c,d\in A$, $c\leq d$,
\begin{equation}\label{cancel.heap}
a\circ \left (b\diamond \underline{c^\diamond\diamond d}\right) = (a\circ b)\diamond \underline{(a\circ c)^\diamond\diamond (a\circ d)}.
\end{equation}
\end{zlist}
Then statement (1) implies (2), and (2) implies (3). If $(A,\diamond)$ has an idempotent, then all above statements are equivalent.
\end{prop}
\begin{proof} 
The implication (1) $\implies$ (2) is obvious (set $\lambda(a,c) = \underline{\sigma(a)^\diamond\diamond(a\circ c)}$). Assume that (2) holds and take any $c\leq d$. Then
$$
a\circ d = a\circ (c\diamond \underline{c^\diamond \diamond d}) = (a\circ c)\diamond \lambda (a, \underline{c^\diamond\diamond d}).
$$
Therefore, 
$$
\lambda (a, \underline{c^\diamond\diamond d}) = \underline{(a\circ c)^\diamond\diamond (a\circ d)}, 
$$
and hence
$$
a\circ \left (b\diamond \underline{c^\diamond\diamond d}\right) = (a\circ b)\diamond  \lambda (a, \underline{c^\diamond\diamond d}) = (a\circ b)\diamond \underline{(a\circ c)^\diamond\diamond (a\circ d)},
$$
as stated in (3).

Assume now that $e$ is an idempotent in $(A,\diamond)$ and define 
\begin{equation}\label{sigma}
\sigma: A\to A, \qquad a \mapsto  a\circ e.
\end{equation}
 Since, for all $a\in A$, $e\diamond a = a$, $e\leq a$. Furthermore,
$$
a = e\diamond \underline{e^\diamond\diamond a} = \underline{e^\diamond \diamond a},
$$
since $e$ is a left identity in $(A,\diamond)$. Assuming that (3) holds we compute, for all $a,b,c\in A$,
$$
a\circ (b\diamond c) = a\circ (b \diamond \underline{e^\diamond\diamond c}) = (a\circ b) \diamond \underline{(a\circ e)^\diamond\di (a\circ c)} = (a\circ b) \diamond \underline{\sigma(a)^\diamond\diamond(a\circ c)}.
$$
Hence statement (3) implies (1), and the equivalence of all three statements is established.
\end{proof}

Proposition~\ref{prop.equiv} is a cancellative semi-truss version of equivalent descriptions of the truss distributive law described in \cite[Theorem~2.5]{Brz:tru} and recalled in Introduction. One of the main differences is, however, that while in the case of a truss the action $\lambda$ determines $\sigma$ uniquely, in the case of a semi-truss, for a fixed $\lambda$ one can define a suitable $\sigma$ using any idempotent of $(A,\di)$. Furthermore and quite understandably, there is no equivalent description of the semi-truss distributive law through $\mu$ or $\kappa$ as in equations \eqref{mu.law} and \eqref{kappa.law} (except for the obvious conclusion that \eqref{lambda.law} implies \eqref{kappa.law}); Proposition~\ref{prop.equiv} leans too heavily on the one-sided cancellation law for that.

\begin{rem}\label{rem.identity}
We observe that in the setup of statement (1) of Proposition~\ref{prop.equiv}, if $(A,\circ)$ has a left identity, say $n$, then $\sigma(n)$ is an idempotent (hence also left identity) in $(A,\diamond)$, since
\begin{align*}
\sigma(n)\diamond\sigma(n) &= n\circ \left(\sigma(n)\diamond\sigma(n)\right)\\
& = (n\circ \sigma(n))\diamond \underline{\sigma(n)^\diamond\diamond (n\circ \sigma(n))}
= \sigma(n)\diamond \underline{\sigma(n)^\diamond\diamond \sigma(n)} = \sigma(n).
\end{align*}
\end{rem}

\begin{prop}\label{prop.sigma}
Let $(A,\diamond,\circ,\lambda)$ be a left semi-truss, let $e$ be an idempotent in the left cancellative semigroup $(A,\diamond)$ and let $\sigma$ be given by \eqref{sigma}. Then:
\begin{zlist}
\item For all $a,b\in A$,
\begin{equation}\label{equivariant}
\sigma(a\circ b) = a\circ \sigma(b).
\end{equation}
\item For all $a,b\in A$,
\begin{equation}\label{sig.cocycle}
\sigma(a\circ b) = \sigma(a)\diamond (a\la \sigma(b)), 
\end{equation}
where $\la$ is given by \eqref{notation.act}.
\item The map $\sigma$ is bijective if and only if $(A,\circ)$ has a right identity with respect to which $e$ is invertible.
\end{zlist}
\end{prop}
\begin{proof}
The first statement follows immediately from the definition of $\sigma$ and the associativity of $\circ$.

To prove statement (2), first note that since $e$ is an idempotent in $(A,\diamond)$ it is also a left neutral element in $(A,\diamond)$, hence, for all $a,b\in A$,
\begin{equation}\label{sig.circ}
a\circ b = a\circ (e\diamond b) = (a\circ e)\diamond \lambda(a, b) = \sigma(a) \diamond (a\la b).
\end{equation}
Combining \eqref{sig.circ} with \eqref{equivariant} we obtain \eqref{sig.cocycle}.

Finally, to prove (3), 
let us assume first that $\sigma$ is a bijective function and define
\begin{equation}\label{e,u}
n = \sigma^{-1}(e),\qquad u = \sigma^{-1}(n).
\end{equation} 
In consequence of  \eqref{equivariant}, for all $a,b\in A$,
\begin{equation}\label{equiv.sig-1}
\sigma^{-1}(a\circ b) = a\circ  \sigma^{-1}( b).
\end{equation}
Therefore, for  all $a\in A$,
$$
a = \sigma^{-1}(\sigma (a)) = \sigma^{-1}(a\circ e) = a\circ n,
$$
where the second equality follows by \eqref{sigma}, and the last one by \eqref{equiv.sig-1} and \eqref{e,u}. Hence $n$ is a right identity for $\circ$. Furthermore, using the above calculation as well as \eqref{equiv.sig-1} and the definition of $u$  in \eqref{e,u}  one finds, for all $a\in A$,
\begin{equation}\label{sigma.inv}
\sigma^{-1} (a) =\sigma^{-1}(a\circ n) = a\circ \sigma^{-1}(n) = a\circ u.
\end{equation}
Consequently,
$$
e\circ u = \sigma^{-1}(e) = n = \sigma(u) = u\circ e.
$$
Therefore, $u$ is the inverse of $e$ with respect to the right identity $n$.

In the converse direction, denote by $u$ the inverse of $e$ with respect to a right identity for $\circ$. Then the inverse of $\sigma$ is given by $\sigma^{-1}(a) = a\circ u.$
 \end{proof}

\begin{prop}\label{prop.semi-brace}
Let $(A,\diamond,\circ, \lambda)$ be a left semi-truss. Assume that $(A,\circ)$ is a group and that $(A,\diamond)$ is a left cancellative semigroup with an idempotent $e$. Define $\sigma:A\to A$ by $\sigma(a) = a\circ e$. Then $\sigma$ is a bijective function, and let $\bullet$ be a binary operation on $A$ given by
$$
a\bullet b = \sigma\left(\sigma^{-1}(a)\circ \sigma^{-1}(b)\right),
$$
for all $a,b,\in A$. The triple $(A,\diamond,\bullet)$ is a (left) semi-brace.
\end{prop}
\begin{proof}
Since $(A,\circ)$ is a group, $e$ is necessarily invertible in $(A,\circ)$, hence $\sigma$ is bijective. Clearly $(A,\bullet)$ is a group, so it remains to check whether the semi-brace law \eqref{semi-brace} binds $\diamond$ with $\bullet$. We take any $a,b,c\in A$ and using \eqref{equivariant}, \eqref{sig.circ} and \eqref{act2} compute,
\begin{align*}
a\bullet (b\diamond c) &= \sigma\left(\sigma^{-1}(a)\circ \sigma^{-1}(b\diamond c)\right) = \sigma^{-1}(a)\circ (b\diamond c)\\
&= \sigma\left(\sigma^{-1}(a)\right)\diamond \left(\sigma^{-1}(a)\la (b\diamond c)\right) \\
&= a\diamond \left(\sigma^{-1}(a)\la b\right)\diamond  \left(\sigma^{-1}(a)\la c\right).
\end{align*}
On the other hand, in view of \eqref{sig.cocycle},
$$
a\bullet b = \sigma\left(\sigma^{-1}(a)\right)\diamond  \left(\sigma^{-1}(a)\la \sigma\left(\sigma^{-1}(b)\right)\right) = a\diamond \left(\sigma^{-1}(a)\la b\right).
$$
Hence, 
$$
a\bullet (b\diamond c) = (a\bullet b)\diamond \left(\sigma^{-1}(a)\la c\right).
$$
Note that $1_\bullet = \sigma(1_\circ) = 1_\circ \circ e = e$, where $1_\circ$ is the identity in $(A,\circ)$ and $1_\bullet$ is the identity of $(A,\bullet)$. By Remark~\ref{rem.identity} (or by the fact that $1_\bullet =e$), $1_\bullet$ is an idempotent hence a left identity in $(A,\diamond)$, so
$$
a\bullet (a^\bullet \diamond c) = (a\bullet a^\bullet) \diamond \left(\sigma^{-1}(a)\la c\right) =1_\bullet\diamond \left(\sigma^{-1}(a)\la c\right) = \sigma^{-1}(a)\la c,
$$
where $a^\bullet$ is the inverse of $a$ in $(A,\bullet)$. Therefore,
$$
a\bullet (b\diamond c) =(a\bullet b)\diamond \left(a\bullet (a^\bullet \diamond c) \right),
$$
and thus the semi-brace law \eqref{semi-brace} is established.
\end{proof}

\begin{rem}\label{rem.semi-brace}
Note that the operation $\bullet$ can be defined whenever $\sigma$ is a bijective function, not only when $(A,\circ)$ is a group. The semigroup operations $\diamond$ and $\bullet$ will be still connected by the semi-brace law \eqref{semi-brace}.
\end{rem}

Recall from \cite{Dri:uns} that a function $r:A\times A\to A\times A$ is said to be a solution a set-theoretic Yang-Baxter equation, if
$$
(r\times \id)(\id \times r)(r\times \id)=(\id \times r)(r\times \id)(\id \times r) .
$$
For a review of some aspects of the solutions of Yang-Baxter equations and their applications the reader can consult \cite{Nic:Yan} and references therein.

Combining Proposition~\ref{prop.semi-brace} with \cite[Theorem~9]{CatCol:sem} we obtain:

\begin{cor}\label{cor.YB}
Let $(A,\diamond,\circ, \lambda)$ be a left semi-truss. Assume that $(A,\circ)$ is a group and that $(A,\diamond)$ is a left cancellative semigroup with an idempotent $e$. Then the function
\begin{align}\label{sol.YB}
r:  A\times A &\to A\times A, \nonumber\\
 (a,b) &\mapsto \left(a\circ e^\circ \left(\left(e\circ a^\circ\circ  e\right)\diamond b\right)\, ,\, e\circ \left(\left(e\circ a^\circ\circ e\right)\diamond b\right)^\circ \circ b\right),
\end{align}
where $a^\circ$ etc., is the inverse of $a$ in $(A,\circ)$, is a solution to the set-theoretic Yang-Baxter equation.
\end{cor}
\begin{proof}
The corollary is a consequence of Proposition~\ref{prop.semi-brace} and \cite[Theorem~9]{CatCol:sem}. By the former, $(A,\diamond,\bullet)$, where 
\begin{equation}\label{bullet.circ}
a\bullet b = \sigma\left(\sigma^{-1}(a)\circ \sigma^{-1}(b)\right) = \sigma^{-1}(a)\circ \sigma^{-1}(b)\circ e = a\circ e^\circ \circ b,
\end{equation}
is a semi-brace. The second of equalities \eqref{bullet.circ} follows by the definition of $\sigma$ in \eqref{sigma}, while the third one follows by \eqref{sigma.inv}. \cite[Theorem~9]{CatCol:sem} states that 
\begin{equation}\label{r.bul}
r:  A\times A \to A\times A,\qquad   r(a,b) = \left(a\bullet\left(a^\bullet\diamond b\right)\, ,\, \left(a^\bullet\diamond b\right)^\bullet \bullet b\right),
\end{equation}
is a solution to the set-theoretic Yang-Baxter equation. In view of \eqref{bullet.circ}, the inverse of $a$ in $(A,\bullet)$ is
$$
a^\bullet = e\circ a^\circ \circ e.
$$
Taking this into account as well as \eqref{bullet.circ}, solution \eqref{r.bul} translates into \eqref{sol.YB}.
\end{proof}

\section{Semi-trusses for inverse semigroups}\label{sec.inverse}
Recall that a semigroup $(A,\diamond)$ is called an {\em inverse semigroup} if any element $a\in A$ admits a unique element $a^\diamond$  such that
\begin{equation}\label{inverse}
a\diamond a^\di \diamond a =a, \qquad a^\di \di a\di a^\di =a^\di.
\end{equation}
Such an $a^\di$ is called the {\em inverse} of $a$. Note that the inverse function $a\mapsto a^\di$ is involutive and satisfies familiar equality
\begin{equation}\label{ab.inv}
(a\di b)^\di = b^\di \di a^\di,
\end{equation}
for all $a,b\in A$.

Conditions \eqref{inverse} imply that both $a^\di\di a$ and $a\di a^\di$ are idempotents. In an inverse semigroup all idempotents commute, hence the set $E(A,\di)$ of idempotents in $(A,\di)$ is a semi-lattice.

An inverse semigroup $(A,\diamond)$ is partially ordered by the relation $\leq$ defined as follows:
\begin{equation}\label{order.inv}
a\leq b \quad \mbox{if and only if there exists  $e\in E(A,\di)$ such that $a = b\diamond e$}.
\end{equation}
Equivalently the order $a\leq b$ \eqref{order.inv} can be defined by requesting the existence of an idempotent $e$ such that $a = e\di b$. The order is compatible with $\di$ and is preserved  by the inverse function $a\mapsto a^\di$. 

Two elements $a,b$ of an inverse semigroup $(A,\di)$ satisfy the {\em left 
compatibility relation}, written $a\sim_l b$, in case $a\di b^\di$ is an idempotent. This relation is reflexive and symmetric, but not transitive. For more informatiom about inverse semigroups the reader is referred to \cite{Law:inv} or \cite{Pet:inv}.

Since inverse semigroups have inverse function, it is natural to take the distributivity between the binary operation $\circ$ and the ternary operation $[a,b,c]=a\di b^\di\di c$ as the starting point for the discussion of trusses in this case. In this context we obtain the following
\begin{prop}\label{prop.inverse}
Let $\circ$ and $\di$ be associative binary operations on a set $A$ such that $(A,\di)$ is an inverse semigroup and consider the following statements:
\begin{zlist}
\item For all $a,b,c,d\in A$,
$$
a\circ \left (b\diamond c^\diamond\diamond d\right) = (a\circ b)\diamond (a\circ c)^\diamond\diamond(a\circ d).
$$
\item There exists function $\sigma:A\times A\to A$ such that, for all $a,b,c\in A$,
\begin{equation}\label{inv.sig.truss}
 a\circ (b\di c) = (a\circ b)\di  \sigma(a,c)^\di\di (a\circ c). 
\end{equation}
\item There exists function $\lambda:A\times A\to A$ such that $(A,\di,\circ,\lambda)$ is a left semi-truss.
\item There exists function $\tau:A\times A\to A$ such that, for all $a,b,c\in A$,
\begin{equation}\label{inv.tau.truss}
a\circ (b\di c) = (a\circ b)\di  \tau(a,b)^\di\di (a\circ c),
\end{equation}
\item There exists function $\mu:A\times A\to A$ such that, for all $a,b,c\in A$, 
\begin{equation}\label{inv.mu.truss}
a\circ (b\di c) = \mu(a, b)\di  (a\circ c).
\end{equation}
\end{zlist}
Then:
\begin{blist}
\item Statement (1) implies all the other statements.
\item Statement (2) implies statement (3).
\item Statement (4) implies statement (5).
\end{blist}
\end{prop} 
\begin{proof}
Clearly, if the assertion (2) holds then $(A,\di,\circ,\lambda)$ is a left semi-truss where $\lambda(a,c) = \sigma(a,c)^\di\di (a\circ c)$. In a similar way assertion (4) implies assertion (5) with $\mu(a,b) = (a\circ b)\di  \tau(a,b)^\di$. 

To prove (a), define 
\begin{equation}\label{sig.inv}
\sigma: A\times A\to A, \qquad (a,c) \mapsto a\circ (c\di c^\di).
\end{equation}
Noting that the involutivity of the inverse function combined with \eqref{ab.inv} implies that $\left(c\di c^\di\right)^\di = c\di c^\di$, we can compute
\begin{align*}
a\circ (b\di c) &= a\circ \left(b\di c\di c^\di\di c\right) = a\circ \left(b\di \left(c\di c^\di\right)^\di \di c\right)\\
& = (a\circ b) \di \left (a\circ \left(c\di c^\di\right)\right)^\di \di (a\circ c) = (a\circ b)\di  \sigma(a,c)^\di\di (a\circ c).
\end{align*}
Hence the equation \eqref{inv.sig.truss} is satisfied and thus statement (1) implies statement (2) and, consequently, also (3)

In a similar way by setting
\begin{equation}\label{tau.inv}
\tau: A\times A\to A, \qquad (a,b) \mapsto a\circ (b^\di\di b),
\end{equation}
one proves that statement (1) implies (4), and thus also (5).
This completes the proof of the proposition.
\end{proof}

\begin{rem}\label{rem.inverse}
Note in passing that functions $\sigma$ defined by \eqref{sig.inv} and $\tau$ defined by  \eqref{tau.inv} are  related by $\tau(a,b) = \sigma(a,b^\di)$.
\end{rem}

\begin{prop}\label{prop.sig.lam}
Let $(A,\di,\circ,\lambda)$ be a left semi-truss such that $(A,\di)$ is an inverse semigroup, and let $\sigma : A\times A\to A$ be given by \eqref{sig.inv}. Then, for all $a,b,c\in A$ and idempotents $e\in E(A,\di)$,
\begin{zlist}
\item $\sigma(a,e) = a\circ e$;\vspace{3pt}

\item $\sigma(a,b\di b^\di) = \sigma(a,b)$;\vspace{3pt}
\item 
$a\circ b = \sigma(a,b)\di \lambda(a,b)$; \vspace{3pt}
\item 
$(a\circ b)\di \lambda(a,b)^\di \leq \sigma(a,b)$ and $\sigma(a,b)^\di\di (a\circ b) \leq \lambda(a,b)$;\vspace{3pt}
\item $(a\circ b)\di (a\circ b)^\di \leq \sigma(a,b)\di \sigma(a,b)^\di$ and $(a\circ b)^\di\di (a\circ b) \leq \lambda(a,b)^\di\di\lambda(a,b)$;\vspace{3pt}
\item 
$a\circ (b\di c) \geq (a\circ b)\di  \sigma(a,c)^\di\di (a\circ c)$;\vspace{3pt}
\item $(a\circ (b\di c))\di (a\circ c)^\di \leq (a\circ b)\di \sigma(a,b)^\di$.
\end{zlist}
\end{prop}
\begin{proof}
Since, for all $e\in E(A,\di)$,  $e^\di = e$, statement (1) follows from the definition of $\sigma$, Eq.~\eqref{sig.inv}. Since $b\di b^\di$ is an idempotent, statement (2) follows from (1) and \eqref{sig.inv}.

Using the definition of the inverse, the truss distributive law in Definition~\ref{def.semi-truss}, and \eqref{sig.inv} we compute, for all $a,b\in A$,
$$
a\circ b = a\circ \left(b\di b^\di\di b\right) = \left(a\circ  \left(b\di b^\di\right)\right)\di \lambda(a,b) = \sigma(a,b)\di \lambda(a,b),
$$
i.e.\ the equality in statement (3). 

Statements (4) and (5) are consequences of (3) and the definition and properties of $\leq$ true in any inverse semigroup. Explicitly, assume that in an inverse semigroup $(A,\di)$, elements $a,b,c$ satisfy  the following equality
\begin{equation}\label{equ}
a\di b = c.
\end{equation}
Then applying $b^\di$ to both sides of \eqref{equ} from the right, and using the fact that $b\di b^\di$ is an idempotent and the definition of partial order $\leq$, we obtain
\begin{equation}\label{ine1}
c\di b^\di \leq a.
\end{equation}
In a similar way, operating with $c^\di$ on both sides of \eqref{equ} and using the fact that $a^\di \di a$ is an idempotent we obtain
\begin{equation}\label{ine2}
a^\di\di c \leq b.
\end{equation}
First applying the inverse function to \eqref{ine1}, then operating on both sides of the resulting inequality by $a$ from the left, and finally using \eqref{equ} we obtain the following  chain of inequalities:
\begin{equation}\label{ine3}
b\di c^\di \leq a^\di \implies a\di b\di c^\di \leq a\di a^\di \implies c\di c^\di \leq a\di a^\di.
\end{equation}
In a similar way, applying the inverse function to \eqref{ine2}, then operating on both sides of the resulting inequality by $b$ from the right, and finally using \eqref{equ} we obtain the following chain of inequalities:
\begin{equation}\label{ine4}
c^\di\di a \leq b^\di \implies c^\di\di a\di b \leq b^\di \di b \implies c^\di\di c \leq b^\di \di b .
\end{equation}
Now making suitable substitutions for $a$, $b$ and $c$ in inequalities \eqref{ine1}--\eqref{ine4}, we conclude that statement (3) implies all the inequalities in (4) and (5)

Using \eqref{semi-truss} and the second statement in (4) we obtain
$$
a\circ (b\di c) = (a\circ b)\diamond \lambda (a, c) \geq (a\circ b)\diamond \sigma(a,c)^\di\di (a\circ c),
$$
i.e. the inequality in statement (6) holds. 

The last statement follows by the semi-truss law \eqref{semi-truss} and the inverse of the second of inequalities (4).
\end{proof}

We finish this note listing a number of properties that $\sigma$ and $\lambda$ satisfy in the case whenever  the inequality in assertion (4) of Proposition~\ref{prop.sig.lam} is an equality (for example, whenever the ternary distributive law in Proposition~\ref{prop.inverse}~(1) holds).

\begin{prop}\label{prop.sig.inv}
Let $(A,\di,\circ,\lambda)$ be a left semi-truss such that $(A,\di)$ is an inverse semigroup and 
\begin{equation}\label{lam.sig.inv}
\lambda(a,b) = \sigma(a,b)^\di\di (a\circ b) = \left(a\circ \left(b\di b^\di\right)\right)^\di \di (a\circ b),
\end{equation}
for all $a,b\in A$ (where $\sigma(a,b) = a\circ (b\di b^\di)$). Then, for all $a,b,c\in A$ and any idempotent $e\in E(A,\di)$:
\begin{zlist}
\item $\sigma\left(a,e\di b\right) = \sigma(a,e)\di\sigma(a,b)^\di\di \sigma(a,b)$; \vspace{3pt}
\item $\sigma\left(a,e\di b\right) = \sigma(a,b)\di\sigma(a,e)^\di\di \sigma(a,e)$; \vspace{3pt}
\item $\sigma(a,b)\di\sigma(a,c)^\di$ and $\sigma(a,c)\di \sigma(a,b)^\di$ are mutually equal idempotents in $(A,\di)$; \vspace{3pt}
\item $\sigma(a,b)\sim_l\sigma(a,c)$; \vspace{3pt}
\item $\sigma(a,b^\di)\di (a\circ b)^\di = \left(a\circ b^\di\right)\di \sigma(a,b)^\di$; \vspace{3pt}
\item $(a\circ b)^\di \di \sigma(a,b) = \sigma(a,b^\di)^\di \di (a\circ b^\di)$; \vspace{3pt}
\item $\lambda (a,b^\di) = \lambda (a,b)^\di$; \vspace{3pt}
\item $\lambda (a,e\di b) = \lambda(a,e)\di \lambda(a,b)$.
\end{zlist}
\end{prop}
\begin{proof}
The assumption means that the distributive law \eqref{inv.sig.truss} holds, and using \eqref{sig.inv}, the fact that  idempotents in an inverse semigroup commute, and statements (1) and (2) of Proposition~\ref{prop.sig.lam} we can compute
\begin{align*}
\sigma\left(a,e\di b\right) &= a\circ (e\di b\di b^\di \di e) = a\circ (e\di b\di b^\di)\\
&= (a\circ e)\di \sigma\left(a,b\di b^\di\right)^\di \di \left(a\circ \left(b\di b^\di\right)\right) = \sigma(a,e)\di\sigma(a,b)^\di\di \sigma(a,b).
\end{align*}
This proves (1). In a similar way and by the same arguments,
\begin{align*}
\sigma\left(a,e\di b\right) &= a\circ (e\di b\di b^\di \di e) = a\circ (b\di b^\di\di e)\\
&=  \left(a\circ \left(b\di b^\di\right)\right) \di \sigma\left(a,e\right)^\di \di (a\circ e)= \sigma(a,b)\di\sigma(a,e)^\di\di \sigma(a,e).
\end{align*}

Combining statements (1) and (2) we obtain the equality
\begin{equation}\label{sig.sig.sig}
 \sigma(a,e)\di\sigma(a,b)^\di\di \sigma(a,b) = \sigma(a,b)\di\sigma(a,e)^\di\di \sigma(a,e).
 \end{equation}
 Multiplying \eqref{sig.sig.sig} from the right by $\sigma(a,b)^\di$ and using the definition of the inverse we obtain
 \begin{align*}
 \sigma(a,e)\di\sigma(a,b)^\di & =  \sigma(a,b)\di\sigma(a,e)^\di\di \sigma(a,e)\di\sigma(a,b)^\di\\
 &= \left(\sigma(a,e)\di\sigma(a,b)^\di\right)^\di\di \left(\sigma(a,e)\di\sigma(a,b)^\di\right).
 \end{align*}
 Since the right hand side of the above equality is an idempotent, so is the left hand side, i.e.\ $ \sigma(a,e)\di\sigma(a,b)^\di \in E(A,\di)$. Consequently, the inverse of this idempotent  $ \sigma(a,b)\di\sigma(a,e)^\di$ in also an idempotent and it is equal to $ \sigma(a,e)\di\sigma(a,b)^\di$.  Replacing $e$ by $c\di c^\di$ and using assertion (2) of Proposition~\ref{prop.sig.lam} one obtains statement (3). The statement (4) is an immediate consequence of (3).
 
 The proof of statement (5) starts with the definition of $\sigma$ to which the law \eqref{inv.sig.truss} is applied, yielding
 $$
 \sigma(a,b^\di) = \left(a\circ b^\di\right)\circ \sigma(a,b)^\di\circ (a\circ b).
 $$
 Multiplying this equality by $(a\circ b)^\di$ we obtain
 $$
  \sigma(a,b^\di) \di (a\circ b)^\di = \left(a\circ b^\di\right)\circ \sigma(a,b)^\di\di (a\circ b)\di (a\circ b)^\di,
  $$
which means
\begin{equation}\label{ineq1}
 \sigma(a,b^\di) \di (a\circ b)^\di \leq  \left(a\circ b^\di\right)\circ \sigma(a,b)^\di.
 \end{equation}
 Since the inequality \eqref{ineq1} holds for all $a,b,\in A$, we can replace $b$ by $b^\diamond$ to observe that
 $$
  \sigma(a,b) \di (a\circ b^\di)^\di \leq  \left(a\circ b\right)\circ \sigma(a,b^\di)^\di .
 $$
 Applying the inverse function to this inequality and remembering that it preserves inequalities we obtain
\begin{equation}\label{ineq2}
  \left(a\circ b^\di\right)\circ \sigma(a,b)^\di \leq \sigma(a,b^\di) \di (a\circ b)^\di .
 \end{equation} 
Combining \eqref{ineq1} with \eqref{ineq2} gives the required equality.
 
 The proof of statement (6) follows similar steps to those taken for proving (5). We start with the definition of $\sigma$ and use \eqref{inv.sig.truss} to obtain
 $$
 \sigma(a,b) = (a\circ b)\di \sigma(a,b^\di)^\di  \di (a\circ b^\di).
 $$
 Then we multiply both sides of this equality by $(a\circ b)^\di$ (this time from the left) to deduce the inequality
 $$
(a\circ b)^\di \di \sigma(a,b) \leq \sigma(a,b^\di)^\di \di (a\circ b^\di).
$$
Replacing $b$ by $b^\di$ and taking the inverse one obtains the opposite relation, and hence the required equality.

Statement (7) follows from (6) and the correspondence between $\lambda$ and $\sigma$ in \eqref{lam.sig.inv}:
\begin{align*}
\lambda(a,b^\di) &= \sigma(a,b^\di)^\di \di (a\circ b^\di) = (a\circ b)^\di \di \sigma(a,b)\\
&=  \left(\sigma(a,b)^\di \di (a\circ b)\right)^\di = \lambda(a,b)^\di,
\end{align*}
as required.

 To prove (8), take any $a,b\in A$ and an idempotent $e$, and using assertion (2), the distributive law \eqref{inv.sig.truss}, assertion (1)  of Proposition~\ref{prop.sig.lam} and (3), compute
 \begin{align*}
\lambda(a,e\di b) &= \sigma(a,e\di b)^\di\di (a\circ (e\di b) )\\
&= \left(\sigma(a,b)\di\sigma(a,e)^\di\di \sigma(a,e)\right)^\di \di (a\circ e)\di \sigma(a,b)^\di\di (a\circ b)\\
&=  \sigma(a,e)^\di\di\sigma(a,e)\di \sigma(a,b)^\di  \di \sigma(a, e)\di \sigma(a,b)^\di\di (a\circ b)\\
&= \sigma(a,e)^\di\di\sigma(a,e)\di \sigma(a,b)^\di  \di (a\circ b) = \lambda(a,e)\di \lambda(a,b),
\end{align*}
as required.
\end{proof}

\subsection*{Acknowledgments}
I would like to express my deep gratitude to Florin Nichita and his family for very warm hospitality in Ploiesti during the 13th International Workshop on Differential Geometry and its Applications at the Petroleum-Gas University, Ploiesti, September 26--28, 2017. I would also like to thank Professor Radu Iordanescu for invitation to give a talk at this workshop.

The research presented in this paper is partially supported by the Polish National Science Centre grant 2016/21/B/ST1/02438.

\end{document}